\documentclass[a4paper,10pt]{amsart}
\usepackage{amssymb}
\usepackage{enumerate}
\usepackage[utf8]{inputenc}
\usepackage{xcolor}
\usepackage[hypertexnames=false]{hyperref}
\usepackage{booktabs}
\usepackage{hyperref,srcltx,comment}

\newcommand{\Z}{{\mathbb{Z}}}
\newcommand{\Q}{{\mathbb{Q}}}
\newcommand{\K}{{\mathbb{K}}}
\renewcommand{\L}{{\mathbb{L}}}
\newcommand{\nK}{n_{\K}}
\newcommand{\nL}{n_{\L}}
\newcommand{\nKL}{n_{\K\L}}
\newcommand{\DK}{\Delta_\K}
\newcommand{\DL}{\Delta_\L}
\newcommand{\lDK}{\ln|\DK|}
\newcommand{\lDL}{\ln|\DL|}
\newcommand{\llDK}{\ln\lDK}

\newcommand{\kK}{\kappa_\K}
\newcommand{\kL}{\kappa_\L}
\newcommand{\OK}{\mathcal{O}_\K}
\newcommand{\OL}{\mathcal{O}_\L}
\newcommand{\dd}{\,\mathrm{d}}
\newcommand{\RR}{\mathcal{R}}
\newcommand{\p}{\mathfrak{p}}
\newcommand{\RD}{\mathfrak{D}}
\DeclareMathOperator{\ExpInt}{E_1}
\DeclareMathOperator{\Norm}{N\!}
\DeclareMathOperator{\GL}{GL}
\DeclareMathOperator{\gal}{gal}

\newtheorem{theorem}{Theorem}
\newtheorem{lemma}{Lemma}
\newtheorem*{corollary*}{Corollary}

\theoremstyle{remark}

\newtheorem*{acknowledgements}{Acknowledgements}

\allowdisplaybreaks[4]

\begin{document}
\title[Bounds for the residue and applications]
      {Bounds for the residue of Dedekind zeta functions under GRH and applications}

\author[S.~R.~Garcia]{Stephan Ramon Garcia}
\address[S.~R.~Garcia]{Department of Mathematics and Statistics\\
         Pomona College\\
         610 N. College Ave.\\
         Claremont, CA 91711\\
         USA}
\email{stephan.garcia@pomona.edu}
\urladdr{\url{https://stephangarcia.sites.pomona.edu/}}

\author[L.~Grenié]{Loïc Grenié}
\address[L.~Grenié]{Dipartimento di Ingegneria gestionale, dell'informazione e della produzione\\
         Università di Bergamo\\
         viale Marconi 5\\
         I-24044 Dal\-mi\-ne\\
         Italy}
\email{loic.grenie@gmail.com}

\author[E.~S.~Lee]{Ethan Simpson Lee}
\address[E.~S.~Lee]{University of the West of England\\
         School of Computing and Creative Technologies\\
         Coldharbour Lane\\
         Bristol BS16 1QY\\
         United Kingdom}
\email{ethan.lee@uwe.ac.uk}
\urladdr{\url{https://sites.google.com/view/ethansleemath/home}}

\author[G.~Molteni]{Giuseppe Molteni}
\address[G.~Molteni]{Dipartimento di Matematica\\
         Università degli studi di Milano\\
         via Saldini 50\\
         I-20133 Milano\\
         Italy}
\email{giuseppe.molteni1@unimi.it}
\urladdr{\url{https://sites.unimi.it/molteni/}}


\date{\today.
}

\begin{abstract}
Let $\kK$ be the residue at $1$ of the Dedekind zeta function of a number field $\K$. We prove explicit
upper and lower bounds for $\kK$ under the Generalized Riemann Hypothesis (GRH). Previous explicit
bounds of comparable generality are weaker and require a weak form of Artin's holomorphy conjecture:
our bounds improve the constants while assuming only GRH for the relevant Dedekind zeta functions. The
new lower bound enables all imaginary quadratic fields with class number $\leq 2000$ to be detected
(assuming GRH). Furthermore, we demonstrate that the upper bound can be significantly reduced when
considering extensions of any fixed, nontrivial field.
\end{abstract}

\maketitle

\section{Introduction}\label{sec:G1}
Let $\K$ be a number field. Dirichlet's class-number formula reveals a deep connection between the
residue $\kK$ of the Dedekind zeta function $\zeta_\K$ at $1$ and algebraic invariants of the number
field $\K$, such as the degree $\nK$, discriminant $\DK$, regulator $R_{\K}$, class number $h_{\K}$, and
number of real and complex embeddings. Consequently, explicit bounds on $\kK$ are desired, both
conditionally (assuming certain well-known conjectures) and unconditionally.
Results in the first class have the shape
\begin{equation}\label{eq:G1}
\frac{c_-}{\llDK}
\leq \kK
\leq c_+(\llDK)^{\nK-1}
\end{equation}
for suitable functions $c_{\pm}$ that depend weakly on $\nK$ or $\DK$.

A first instance of such a result is due to Littlewood \cite{Littlewood3}, who proved that for any
quadratic field $\K=\Q(\sqrt{\Delta_\K})$ with discriminant $\Delta_\K < 0$ one has
\begin{equation}\label{eq:G2}
c_- = \zeta(2)/(2e^{\gamma}) + o(1)
\qquad\text{and}\qquad
c_+ = 2e^\gamma + o(1)
\end{equation}
as $\Delta_\K\to -\infty$, in which $\zeta = \zeta_{\Q}$ denotes the Riemann zeta function, under the
assumption that the associated Dirichlet $L$-function has no zeros with real part $>1/2$. %
Under the same hypothesis, Littlewood also proved that these bounds are optimal up to a factor $1/2$.

In 2003, as part of a related work, Bessassi \cite[Thm.~10]{Bessassi} extended Littlewood's lower bound
and proved that if $\K$ is normal and $\zeta_\K$ satisfies the Generalized Riemann Hypothesis (GRH), then
the lower bound in \eqref{eq:G1} holds with
\begin{equation*}
c_- = 1/(2e^\gamma) + o(1).
\end{equation*}
In this result the little-$o$ term is not made explicit, but an examination of the proof shows that it is
essentially of order $\ln\ln\ln|\Delta_\K|/\ln\ln|\Delta_\K|$.

More recently, Cho and Kim~\cite{ChoKim} conditionally extended Littlewood's and Bes\-sas\-si's results to
every number field $\K$, proving the claim with
\begin{equation}\label{eq:G3}
c_- = \zeta(\nK)/(2e^{\gamma}) + o(1)
\qquad\text{and}\qquad
c_+ = (2e^\gamma + o(1))^{\nK-1},
\end{equation}
in which the little-$o$ terms are not made explicit but go to $0$ as $|\DK|\to\infty$.
This result generalizes~\eqref{eq:G2} and is proved with the same main ideas. In particular, Cho and Kim
assumed the validity of GRH for the relevant Artin $L$-functions and a strong form of Artin's conjecture
saying that the quotient $\zeta_\K/\zeta$ is an entire function because it may be interpreted as the
$L$-function of a suitable automorphic representation. In the same paper, the authors also proved that
the stronger bounds
\begin{equation}\label{eq:G4}
c_- = \zeta(\nK)/e^{\gamma} + o(1)
\qquad\text{and}\qquad
c_+ = (e^\gamma + o(1))^{\nK-1}
\end{equation}
hold for ``almost every field'' (the exceptions are quantified in a quite technical way), at least when
we restrict to fields of degree $3$, $4$ and $5$ having full symmetric Galois closure.
For degrees $3$ and $4$ the result is unconditional, while for degree $5$ they assume that the quotient
$\zeta_\K/\zeta$ can be realized as an Artin $L$-function coming from an automorphic representation for
$\GL_4$. Under the same hypotheses they also prove that these stronger bounds are attained by suitable
sequences of fields so that no further improvement is possible for all fields of that type.

Totally explicit bounds are also known. For example, the first and third authors \cite{GarciaLee3} and
Palojärvi and Simonič \cite[Cor.~4]{PalSim} proved, respectively, that if GRH holds for $\K$ and $\Q$ and
if $\zeta_\K/\zeta$ is entire, then
\begin{equation}\label{eq:G5}
c_- = e^{- 17.81(n_{\K} - 1)}
\quad\text{and}\quad
c_+ = \big(2e^\gamma e^{\frac{2.475}{\ln\ln{|\Delta_{\K}|}}}\big)^{n_{\K}-1}
\end{equation}
are admissible in \eqref{eq:G1} when $|\Delta_{\K}| \geq 5.4\cdot 10^{6}$. The results in
\cite{GarciaLee3} hold on a much broader range of $|\Delta_{\K}|$. All these results have been refined by
the present authors in \cite{GarciaLee4}, where it is proved under the same assumptions that
\begin{align}
c_-
&= \frac{\zeta(n_{\K})}{2e^\gamma}
   (\ln\ln{|\Delta_{\K}|})^{\frac{-19 (n_{\K} - 1)}{\ln\ln{|\Delta_{\K}|}}}      \label{eq:G6}
\intertext{and}
c_+
&= \Big(2
  e^{\gamma}
  (\ln\ln{|\Delta_{\K}|})^{\frac{19}{\ln\ln{|\Delta_{\K}|}}}\Big)^{n_{\K} - 1}   \label{eq:G7}
\end{align}
are admissible in \eqref{eq:G1} for every field. In this case, both formulas for $c_{\pm}$ agree
with~\eqref{eq:G3} as the discriminant diverges.
\medskip

In this paper we approach the problem from a different point of view. Specifically, the computation of
$\kK$ with a given precision is a key ingredient for the application of Buchmann's algorithm computing the
class group and the regulator of $\K$. For this reason, explicit formulas giving $\kK$ have been devised.
One of them is in Belabas and Friedman~\cite{BelabasFriedman}, and a different formula is in Grenié and
Molteni~\cite[Section~2, Corollary~2.4]{GrenieMolteni2}. Belabas--Friedman's formula is complicated, but
Grenié--Molteni's is simpler and, after further simplifications and a suitable choice of a parameter, it
can be used to deduce bounds for $c_+$ and $c_-$. This is our primary result.
\begin{theorem}\label{th:G1}
Let $\K\neq \Q$ be a number field. Assume GRH for the Dedekind zeta functions of $\K$ and $\Q$.
Then the residue $\kK$ of the Dedekind zeta function $\zeta_{\K}$ of $\K$ at $1$ satisfies
\begin{equation*}
\frac{c_-}{\llDK}
\leq \kK
\leq c_+(\llDK)^{\nK-1}
\end{equation*}
with
\begin{align}
c_- &= \frac{\zeta(\nK)}{2e^{\gamma}}\exp\Bigl(-\frac{3}{\llDK}\Bigr)  \quad \text{and} \label{eq:G8}\\
c_+ &= \big(2e^\gamma\big)^{\nK-1}\exp\Big(\frac{11 - 3\nK}{4\llDK}\Big).               \label{eq:G9}
\end{align}
\end{theorem}
Theorem \ref{th:G1} is proved under GRH alone, without any form of Artin's conjecture. %
In contrast with \eqref{eq:G5}, \eqref{eq:G6}, and \eqref{eq:G7}, one sees that the new remainder term 
for $c_-$ no longer depends on $n_{\K}$, that $c_+$ can be simplified to $(2e^\gamma)^{\nK-1}$ for all 
$\nK \geq 4$, and that for quadratic and cubic fields our $c_+$ is also better than \eqref{eq:G5} and 
\eqref{eq:G7}.

In our computations we take care to simplify the conclusion while reducing the size of all diverging
terms. With some extra care one could further reduce the secondary parameters in~\eqref{eq:G8}
and~\eqref{eq:G9}.

Building upon Theorem \ref{th:G1}, we make further progress toward two more problems. First, we apply the
lower bound in Theorem \ref{th:G1} to a generalization of Gauss' class number problem. Second, we
demonstrate how one can refine the upper bound in Theorem \ref{th:G1} by appealing to the subfields of
the number field in question.

\subsection*{An application for the lower bound in Theorem~\ref{th:G1}}
The determination of all imaginary quadratic fields with a given class number is a classical motivation for
the study of lower bounds of residues of their Dedekind zeta function. Gauss' class number problem is the
archetypical example of such computations.  It was solved by Heegner \cite{Heegner}, Stark \cite{Stark67,
Stark69}, and Baker \cite{Baker} in the 1960s.
Then, Oesterlé \cite{Oesterle1,Oesterle2}, Goldfeld \cite{Goldfeld1,Goldfeld2}, and Gross--Zagier
\cite{GrossZagier} paved the way to classify all imaginary quadratic fields with class number up to a
given threshold, at least from a theoretical point of view. By elaborating on their ideas and making
several technical improvements which reduce the required computations, Watkins \cite{Watkins} determined
all such fields with class number $\leq 100$. These results are unconditional, that is, they are proved
without assuming any unproved hypotheses.

The lower bound in \eqref{eq:G8} can be used to prove the following result.
\begin{corollary*}
Let $\K$ be any imaginary quadratic field. Assume GRH for the Dedekind zeta function of $\K$. Then the
class number $h_{\K}$ satisfies the inequality
\begin{equation*}
h_\K \geq \frac{\pi}{12 e^\gamma} \frac{\sqrt{|\DK|}}{\llDK} e^{-\frac{3}{\llDK}}.
\end{equation*}
\end{corollary*}
\begin{proof}
The lower bound is less than $1$ for $|\DK|\leq 3\,552$.
%
%
The other imaginary quadratic fields satisfy $\kK \sqrt{|\Delta_{\K}|} = \pi h_{\K}$ and $\zeta(\nK) =
\pi^2/6$, so~\eqref{eq:G8} proves the statement.
\end{proof}
Using the corollary we see that if $h_\K \leq h_0$, then $|\Delta_\K| \leq \Delta_0$ for each pairing
\begin{equation*}
(h_0, \Delta_0) \in\{(1, 3\,552), (100, 30\,864\,538), (2\,000, 12\,338\,970\,045)\} .
\end{equation*}
These bounds can be further improved to
\begin{equation}\label{eq:G10}
(h_0, \Delta_0) \in\{(1, 1\,592), (100, 19\,124\,734), (2\,000, 8\,213\,527\,466)\}
\end{equation}
with a variation of the formula that we will show during the proof of~\eqref{eq:G8}.
%
%
%
%
A similar lower bound is proved in~\cite{LamzouriLiSoundararajan}, Corollary~1.3.
%
Its conclusion is actually stronger than our corollary but it assumes that $|\Delta_\K| \geq 10^{10}$ 
while in our case all values are admitted. Our corollary improves considerably in case we assume from the
beginning that $|\Delta_\K|$ is large, and it becomes better and better as the minimum discriminant
increases. However, it remains (barely) weaker than what is proved in~\cite{LamzouriLiSoundararajan}: the
ad hoc tools in~\cite{LamzouriLiSoundararajan} cannot be beaten by our general tools.

In 2004, Watkins~\cite{Watkins} published a full computation of all complex quadratic fields with class
number up to $100$. In 2006 Jacobson, Ramachandran, and Williams~\cite{JacobsonRamachandranWilliams}
explored the class number of all complex quadratic fields up to $10^{11}$, and in 2016 Mosunov and
Jacobson~\cite{MosunovJacobson} extended this computation to $2^{40}\approx 10^{12}$. In both cases, the
main purpose of the authors was to check the validity of Cohen--Lenstra conjectures. In particular, only
aggregated tables of values are explicitly given, and the raw values of $h_\K$ are not made available.
With the help of our corollary, we have classified the imaginary quadratic fields up to our bound for
$2\,000$, i.e. up to $8\,213\, 527\,466$ (slightly better than what is necessary if we use the bound
in~\cite{LamzouriLiSoundararajan}).
%
%

Table \ref{tab:G1} collects some information about the number of fields and the maximum discriminant in
that range. A full description of all fields with class number up to $2000$ is available at
\url{https://sites.unimi.it/molteni/research/class_numbers/class_numbers.html}.

\begin{table}[ht]
\centering
\begin{tabular}{rrr|rrr}
  \toprule
    \multicolumn{1}{c}{$h_0$} & \# of fields & \multicolumn{1}{c}{Max disc.}
   &\multicolumn{1}{c}{$h_0$} & \# of fields & \multicolumn{1}{c}{Max disc.}       \\
  \midrule
     100 &        42272 &   2383747   &       1100 &       4972445  &  283506787   \\
     200 &       167157 &   8377363   &       1200 &       5941733  &  351036667   \\
     300 &       372393 &  19225243   &       1300 &       6960178  &  398158363   \\
     400 &       665008 &  35723227   &       1400 &       8061896  &  480668203   \\
     500 &      1028749 &  55171603   &       1500 &       9245562  &  562394347   \\
     600 &      1488496 &  81792547   &       1600 &      10553429  &  643541347   \\
     700 &      2014583 & 114148483   &       1700 &      11881764  &  707327707   \\
     800 &      2645480 & 138731107   &       1800 &      13348410  &  805169467   \\
     900 &      3342720 & 199362403   &       1900 &      14824012  &  898772323   \\
    1000 &      4115897 & 227932027   &       2000 &      16453493  & 1062350083   \\
  \bottomrule
\end{tabular}
\medskip
\caption{The number of imaginary quadratic fields $\K$ with $h_\K\leq h_0$ and their maximum discriminant
(in absolute value), for some $h_0 \leq 2\,000$.} %
\label{tab:G1}
\end{table}
%
%

\subsection*{Refinements to the upper bound in Theorem~\ref{th:G1}}
In our proof of Theorem~\ref{th:G1}, the exponent $\nK-1$ of $\llDK$ in the upper bound comes from a
certain procedure that treats all primes in $\Z$ as totally decomposed prime ideals in $\OK$. This is far
from the truth, since the density of this family of primes is only $1/[\K^{\gal}:\Q]$ (equal to $1/\nK$
when $\K/\Q$ is Galois). In less algebraic terms, this procedure essentially estimates $\zeta_{\K}$ with
the $\nK$-th power of $\zeta$, producing a $\nK$-order pole where actually we only have a simple pole.
This step is in some sense a necessity of our procedure, but we {believe} that it greatly overestimates
that exponent, so it is reasonable to expect that the upper bound can have a smaller exponent. On the
other hand, Cho and Kim's result~\eqref{eq:G4} shows that the bounds~\eqref{eq:G1} are essentially the
correct ones when we look for something that is true for all fields. Thus, if a smaller exponent is
possible, it must be for families of fields but not for all of them. The following result shows that this
phenomenon really happens for fields having a fixed and nontrivial subfield.
\begin{theorem}\label{th:G2}
Let $\Q \subseteq \L \subseteq \K$ with $\K \neq \Q$, and let $\nKL$ be the intermediate degree $[\K:\L]$.
Assume GRH for the Dedekind zeta functions of $\Q$, $\L$ and $\K$. Then, we have
\begin{equation*}
\kK \leq \kL^{\nKL} (2e^\gamma \llDK)^{\nKL-1} \exp\Big(\frac{3}{\llDK}\Big).
\end{equation*}
\end{theorem}
Note that when $\L=\Q$ the result is weaker than Theorem~\ref{th:G1} and that when $\L=\K$ the
statement is a tautology, so the real interest of this result is for the case $\Q \subsetneq \L
\subsetneq \K$.

As announced, Theorem \ref{th:G2} shows that when $\K$ ranges only among the family of fields containing
a fixed field $\L$, the order of divergence of $\kK$ in terms of $\llDK$ is $\nKL-1$ at most, and hence
much smaller than $\nK-1$.
\smallskip

Finally, we notice that combining Theorems~\ref{th:G1} and~\ref{th:G2} we get the following conclusion
for towers of fields.
\begin{corollary*}
Let $\Q= \K_0\subsetneq \K_1\subsetneq \K_2\subsetneq \ldots \subsetneq \K_\ell$ be any tower of distinct
number fields. To simplify notation let $n_j=n_{\K_j}=[\K_j:\Q]$, $\kappa_j=\kappa_{\K_j}$, and $\Delta_j =
\Delta_{\K_j}$ for $j=0,1,\ldots,\ell$. Assume GRH for all involved fields. Then
\[
\kappa_\ell
\leq (2e^\gamma)^{n_\ell-1}
     \prod_{j=1}^{\ell} (\ln\ln|\Delta_j|)^{\frac{n_{\ell}}{n_{j-1}} - \frac{n_\ell}{n_j}}
     \exp\bigg(\frac{11 n_\ell/n_1 - 3n_\ell}{4\ln\ln|\Delta_1|}
              + \sum_{j=2}^\ell\frac{3n_\ell/n_j}{\ln\ln|\Delta_j|}
         \bigg).
\]
In particular, for $n_1 \geq 8$ we get
\[
\kappa_\ell
\leq (2e^\gamma)^{n_\ell-1}
     \prod_{j=1}^{\ell} (\ln\ln|\Delta_j|)^{\frac{n_{\ell}}{n_{j-1}} - \frac{n_{\ell}}{n_j}}
\]
for every $\ell$ (and $n_1\geq 4$ suffices for $\ell=1$, $n_1\geq 6$ for $\ell=2$, and $n_1\geq 7$ for
$\ell=3$).
\end{corollary*}
%
The sum $\sum_{j=1}^{\ell}\big(\frac{n_{\ell}}{n_{j-1}} - \frac{n_{\ell}}{n_j}\big)$ is $n_\ell - 1 =
n_{\K_\ell}-1$, so the corollary shows how the exponent $\nK-1$ of $\llDK$ in Theorem \ref{th:G1} is now
distributed among the analogous exponents of intermediate fields.
\begin{proof}
The proof is by induction on $\ell$, the length of the tower of fields. The claim for $\ell=1$ holds
since it coincides with what we know from Theorem~\ref{th:G1}. 

Next, let $\ell\geq 2$ and assume that the claim
holds for $\ell-1$. Then the formula for the tower $\Q\subsetneq \K_1\subsetneq \K_2\subsetneq \ldots
\subsetneq \K_{\ell-1}$ produces an upper bound for $\kappa_{\ell-1}$ that produces the desired bound for
$\kappa_\ell$ via the extension
$\Q\subsetneq \K_{\ell-1}\subsetneq \K_{\ell}$ and Theorem~\ref{th:G2}. 
By the relative discriminant formula, $|\Delta_j| = |\Delta_1|^{n_j/n_1} \Norm_{\K_1/\Q}
(\RD_{\K_j/\K_1})$, where $\RD_{\K_j/\K_1}$ denotes the relative discriminant ideal of $\K_j/\K_1$. It
follows that $|\Delta_j|\geq |\Delta_1|$ for every $j$. As a consequence, the extra exponential term is
bounded by
\[
\frac{11\frac{n_{\ell}}{n_1}-3n_{\ell} + 12\sum_{j=2}^{\ell}n_{\ell}/n_j}{4\ln\ln|\Delta_1|},
\]
which is negative when $3n_1 > 11 + 12\sum_{j=2}^{\ell}n_1/n_j$. We are assuming that the fields are
distinct, hence $n_1/n_j \leq 1/2^{j-1}$, so the exponent is negative when $3 n_1 > 11 + 12$, which holds 
when $n_1 \geq 8$. For $\ell=1$ (respectively, $2$, $3$) the sum is $0$ (respectively, at most $1/2$, 
$3/4$) and the special cases follow.
\end{proof}
\begin{acknowledgements}
SRG acknowledges the support of NSF Grant DMS-2452084. LG and GM are members of the INdAM group GNSAGA.
Computations were performed using \texttt{PARI/GP}~\cite{PARI2} using the PlaFRIM experimental
testbed, supported by Inria, CNRS (LABRI and IMB), Université de Bordeaux, Bordeaux INP and Conseil
Régional d’Aquitaine (see \url{https://www.plafrim.fr}). The authors thank Karim Belabas for his
assistance for these computations.
\end{acknowledgements}

\section{Preparatory Lemmas}\label{sec:G2}
In what follows, $\OK$ denotes the ring of integers of a number field $\K$, $\p$ a nonzero prime ideal of
$\OK$, and $p$ a (rational) prime number. The von Mangoldt function $\Lambda_\K$ is defined on the set of
ideals of $\OK$ by $\Lambda_\K(\mathfrak{I}) = \ln\Norm\p$ if $\mathfrak{I}=\p^m$ for some $\p$ and
$m\geq 1$, and it is zero otherwise. Let
\begin{equation*}
\tilde{\Lambda}_\K(n)
 := \sum_{\substack{\mathfrak{I}\subset \OK\\ \Norm\mathfrak{I} = n}} \Lambda_\K(\mathfrak{I}).
\end{equation*}
As usual in analytic number theory, we suppress the subscript for functions related to
the rational field, so that $\zeta=\zeta_\Q$ is the Riemann zeta function and $\Lambda=\tilde{\Lambda}_\Q$
is the classical von Mangoldt function.

Before we can proceed to the proof of the theorems, we require six preparatory lemmas. The first lemma
is~\cite[Corollary~2.4]{GrenieMolteni2}.
\begin{lemma}\label{lem:G1}
Assume that $\zeta_{\K}$ and $\zeta$ satisfy GRH. Let $f(x)=(x\ln x)^{-1}$, $y=1/\ln N$,
$\ExpInt(x)= \int_1^{+\infty} e^{-xt}t^{-1}\dd t$,
\begin{align*}
 \alpha_\K &= 0.5375 \lDK - 1.0355 \nK + 5.4341,  &&& \beta_\K  &= \nK - 1,              \\
 \gamma_\K &= 1.0155 \lDK - 2.1041 \nK + 10.1799, &&& \delta_\K &= \lDK - 1.415  \nK + 4 .
\end{align*}
For $N\geq 3$, we have
\begin{equation*}
\ln \kK
= \sum_{n\leq N} \big(\tilde{\Lambda}_\K(n)-\Lambda(n)\big) \big(f(n) - f(N) - (n-N)f'(N)\big) + \RR(N) ,
\end{equation*}
in which
\begin{align*}
 |\RR(N)|
  &\leq   \alpha_\K \bigg(\frac{\frac{5}{2}+y}{\sqrt{N}\ln N} + \frac{3}{4} \ExpInt\Big(\frac{\ln N}{2}\Big)\bigg)
        + \beta_\K  \frac{2+3y}{N}
        + \gamma_\K \frac{2y+y^2}{N}
        + \delta_\K \frac{y+y^2}{N^2} .
\end{align*}
\end{lemma}
\begin{lemma}\label{lem:G2}
Assume GRH as in the previous lemma. Then for $N\geq 121$, we have
\begin{equation*}
|\RR(N)|\leq \frac{2.5 \lDK}{\sqrt{N}\ln N} - \frac{3.5 \nK}{\sqrt{N}\ln N} + \frac{23.5}{\sqrt{N}\ln N}.
\end{equation*}
\end{lemma}
\begin{proof}
Integrate by parts three times in the definition of $\ExpInt(x)$ and obtain
\begin{equation*}
\ExpInt(x) \leq \frac{e^{-x}}{x}\Big(1-\frac{1}{x}+\frac{2}{x^2}\Big).
\end{equation*}
Thus,
\begin{equation*}
\ExpInt(\tfrac{1}{2}\ln N) \leq \frac{2}{\sqrt{N}\ln N}(1-2y+8y^2)
\end{equation*}
with $y=1/\ln N$, and from Lemma~\ref{lem:G1} we get
\begin{align*}
|\RR(N)|
\leq \alpha_\K \frac{4-2y+12y^2}{\sqrt{N}\ln N}
     + \beta_\K  \frac{2+3y}{N}
     + \gamma_\K \frac{2y+y^2}{N}
     + \delta_\K \frac{y+y^2}{N^2}.
\end{align*}
Since $\gamma_\K \leq 2\alpha_\K$ and $\delta_\K \leq 2\alpha_\K + \nK$, we obtain
\begin{align*}
\sqrt{N}\ln N|\RR(N)|
&\leq A\alpha_\K + B \nK - C                                           \\
& =   0.5375\cdot A\,\lDK - (1.0355\cdot A - B) \nK + 5.4341\cdot A - C,
\end{align*}
in which
\begin{align*}
A  &:= 4 - 2y + 12y^2
      + \frac{4+2y}{\sqrt{N}}
      + \frac{2+2y}{N^{3/2}},        \\
B  &:= \frac{2\ln N+3}{\sqrt{N}}
      + \frac{\ln N +1}{N^{3/2}\ln N} ,
\qquad\text{and}\qquad
C  := \frac{2\ln N+3}{\sqrt{N}}.
\end{align*}
The three coefficients $0.5375\cdot A$, $B - 1.0355\cdot A$ and $5.4341\cdot A - C$ decrease when $N \geq
121$, and we have
\begin{equation*}
\sqrt{N}\ln N|\RR(N)| \leq 2.5 \lDK - 3.5 \nK + 23.5,
\end{equation*}
which is the claim.
\end{proof}
%
%
%
%
%

Next, we prove bounds for
\begin{gather*}
\sum_{n\leq N} \tilde{\Lambda}_\K(n) f(n),
\quad
\sum_{n\leq N}\tilde{\Lambda}_\K(n)[f(N) + (n-N)f'(N)],\\
\text{and}
\quad
\sum_{n\leq N} \Lambda(n)\big(f(n) - f(N) - (n-N)f'(N)\big)
\end{gather*}
in Lemmas~\ref{lem:G3}, \ref{lem:G4}, and \ref{lem:G5} respectively.
\begin{lemma}\label{lem:G3}
If $\K\neq \Q$ and $N \geq 49$, then
\begin{align*}
\sum_{n\leq N} \tilde{\Lambda}_\K(n) f(n)
 \geq \ln\zeta(\nK) - \frac{4}{\sqrt{N}\ln N}.
\end{align*}
\end{lemma}
\begin{proof}
First recall that $f(n) = (n\ln n)^{-1}$ and observe that
\begin{align*}
\sum_{n\leq N}\frac{\tilde{\Lambda}_\K(n)}{n\ln n}
 = \sum_{\substack{\p,m\\\Norm\p^m\leq N}}\frac{\ln\Norm\p}{\Norm\p^m (m\ln\Norm\p)}
 = \sum_{\substack{\p,m\\\Norm\p^m\leq N}}\frac{1}{m\Norm\p^m}.
\end{align*}
For a prime ideal $\p$, let $p$ be the prime below $\p$ and let $\ell=\ell(\p)$ denote its inertia
degree, so that $\Norm\p=p^\ell$. Group the prime ideals according to the prime below them and get
\begin{align*}
\sum_{n\leq N}\frac{\tilde{\Lambda}_\K(n)}{n\ln n}
 = \sum_{p\leq N}\sum_{\substack{\p,m\\p^{\ell m}\leq N}}\frac{1}{mp^{\ell m}}.
\end{align*}
For each prime $p$, there is at least one prime ideal $\p$ such that $\p \mid p$, and $\ell(\p)\leq \nK$
in any case. Therefore,
\begin{align*}
\sum_{n\leq N}\frac{\tilde{\Lambda}_\K(n)}{n\ln n}
&\geq \sum_{\substack{p,m\\p^{m\nK}\leq N}}\frac{1}{mp^{m\nK}}
 =    \ln\zeta(\nK) - \sum_{\substack{p,m\\p^{m\nK} > N}}\frac{1}{mp^{m\nK}}.
\end{align*}
Let $\Pi(x) = \pi(x) + \frac{1}{2}\pi(x^{1/2}) + \frac{1}{3}\pi(x^{1/3}) + \cdots $.
Then, by partial summation
\[
\sum_{\substack{p,m\\p^m > N^{1/\nK}}}\frac{1}{mp^{m\nK}}
= \int_{N^{1/\nK}}^\infty \frac{\dd \Pi(x)}{x^{\nK}}
= -\frac{\Pi(N^{1/\nK})}{N} + \nK\int_{N^{1/\nK}}^\infty \frac{\Pi(x)}{x^{\nK+1}}\dd x.
\]
We know that $\Pi(x) \leq 1.5\frac{x}{\ln x}$ for every $x \geq 2$ and $\Pi(x)\geq \frac{x}{\ln x}$ for
every $x\geq 7$: This may be verified directly when $x\leq 500$ while for $x\geq 500$ they follow
respectively by the upper bounds $\Pi(x) \leq \pi(x) + \frac{1}{2}\pi(\sqrt{x}) + \frac{\ln x}{3\ln
2}\pi(\sqrt[3]{x})$ and the $\pi(x)\leq 1.25506 \frac{x}{\ln x}$ in \cite[Corollary 1]{RosserSchoenfeld}
(this same inequality is applied also for $\pi(\sqrt{x})$ and $\pi(\sqrt[3]{x})$), and by the lower
bounds $\Pi(x) \geq \pi(x)$ and $\pi(x)\geq x/\ln x$ again in \cite[Corollary 1]{RosserSchoenfeld}.
%
%
%
\noindent %
Thus, in case $N\geq 7^{\nK}$ we have
\begin{align*}
\sum_{\substack{p,m\\p^{m} > N^{1/\nK}}}\!\!\!\!\frac{1}{mp^{m\nK}}
&\leq  -\frac{N^{1/\nK}}{N\ln(N^{1/\nK})} + \frac{3}{2}\nK\int_{N^{1/\nK}}^\infty \frac{\dd x}{x^{\nK}\ln x}\\
&\leq  -\frac{N^{1/\nK}}{N\ln(N^{1/\nK})} + \frac{3}{2}\frac{\nK}{\nK-1}\frac{N^{1/\nK}}{N\ln(N^{1/\nK})}
 \leq  \frac{2 N^{1/\nK}}{N\ln(N^{1/\nK})}.
\end{align*}
The function $x/\ln x$ increases for $x\geq e$, and $\nK\geq 2$, hence the bound is smaller than what we
get for $\nK=2$, that is,
\[
\sum_{\substack{p,m\\p^{m} > N^{1/\nK}}}\frac{1}{mp^{m\nK}}
 \leq  \frac{4}{\sqrt{N}\ln N}.
\]
Assume now that $N < 7^{\nK}$. In this case we deduce the conclusion from the elementary upper bound
given by the corresponding sum over all integers. In fact,
\begin{align*}
\sum_{\substack{p,m\\p^{m} > N^{1/\nK}}}\!\!\!\!\frac{1}{mp^{m\nK}}
&\leq \sum_{k > N^{1/\nK}}\frac{1}{k^{\nK}}
 \leq \int_{N^{1/\nK}}^\infty\frac{\dd t}{t^{\nK}}
 =    \frac{1}{\nK-1}\frac{N^{1/\nK}}{N},
\end{align*}
so that it is sufficient to prove that
\[
\frac{1}{\nK-1}\frac{N^{1/\nK}}{N}
\leq \frac{4}{\sqrt{N}\ln N},
\quad
\text{i.e.}
\quad
\frac{1}{\nK-1} \frac{\ln N}{N^{1/2-1/\nK}} \leq 4.
\]
Since we are assuming that $49\leq N<7^{\nK}$, this is implied by the inequality $\frac{\nK}{\nK-1}
\frac{\ln 7}{49^{1/2-1/\nK}} \leq 4$ which holds true for every $\nK \geq 2$.
\end{proof}
Before we proceed to the statement and proof of Lemma~\ref{lem:G4}, observe that
\begin{equation}\label{eq:G11}
\sum_{n\leq N}\tilde{\Lambda}_\K(n)[f(N) + (n-N)f'(N)]
=     \frac{\psi_\K(N)}{N\ln N} + \Big(1+\frac{1}{\ln N}\Big)\frac{\psi_\K^{(1)}(N)}{N^2\ln N},
\end{equation}
in which
\begin{equation*}
\psi_\K(x)
= \sum_{\Norm\mathfrak{I}\leq x}\Lambda_\K(\mathfrak{I})
= \sum_{n\leq x}\tilde{\Lambda}_\K(n)
\quad\text{and}\quad
\psi^{(1)}_\K(x) = \int_2^x \psi_\K(u)\dd u.
\end{equation*}
Observe also that $\psi=\psi_\Q$, and we further denote $\psi^{(1)}(x) = \psi^{(1)}_\Q(x) =
\int_2^x\psi(u)\dd u$. The standard bound $\tilde{\Lambda}_\K(n)\leq \nK\Lambda(n)$ implies that
\begin{equation*}
\psi_\K(x) \leq \nK\psi(x)
\qquad\text{and}\qquad
\psi^{(1)}_\K(x) \leq \nK\psi^{(1)}(x),
\end{equation*}
so the desired bound can be deduced from upper bounds for $\psi(x)$ and $\psi^{(1)}(x)$, such
as~\cite[Theorem~10]{Schoenfeld1} or \cite[Theorem~1.1]{LeeNosal} and~\cite[Theorem~1.1]{GrenieMolteni2},
respectively.
%
%
The bound we obtain in this way has the form $c \frac{\nK}{\ln N}$ for a suitable constant $c$ (in our
computation, one can take $c=2$ when $N\geq 121$). However, a $\sqrt{N}$ can be gained in the denominator
by directly estimating $\psi_\K$ and $\psi^{(1)}_\K$.
%
\begin{lemma}\label{lem:G4}
Assume that $\zeta_{\K}$ satisfies GRH and let $N \geq 121$. Then
\begin{align*}
\sum_{n\leq N}\tilde{\Lambda}_\K(n)[f(N) + (n-N)f'(N)]
 \leq  \frac{3.2 \lDK}{\sqrt{N}\ln N}
     + \frac{4 \nK}{N\ln N}
     + \frac{3.4}{\ln N}.
\end{align*}
\end{lemma}
\begin{proof}
For $\psi^{(1)}_\K$, we use \cite[Theorem~1.1]{GrenieMolteni2}, which reads:
\begin{equation*}
\psi^{(1)}_\K(x)
\leq \frac{x^2}{2}
     + \alpha'_\K x^{3/2}
     + \beta_\K  x\ln x
     + \gamma'_\K x
     + \delta_\K
\qquad\forall x\geq 3,
\end{equation*}
in $\alpha'_\K\leq \alpha_\K$, $\gamma_\K' \leq \gamma_\K$ and $\alpha_\K$, $\beta_\K$, $\gamma_\K$ and
$\delta_\K$ are as in Lemma~\ref{lem:G1}. The inequalities $\beta_\K\leq \nK$, $\gamma_\K' \leq
2\alpha_\K$ and $\delta_\K \leq 2\alpha_\K + \nK$ ensure that
%
\begin{align*}
\psi^{(1)}_\K(N)
&\leq \frac{N^2}{2}
    + (N^{3/2} + 2N + 2)\big(0.5375 \lDK + 5.4341\big) \\
&\quad
    + \big(N\ln N + 1 - 1.0355(N^{3/2} + 2N + 2)\big)\nK
\end{align*}
so that
\begin{equation}\label{eq:G12}
\frac{1+1/\ln N}{N^2\ln N}\psi^{(1)}_\K(N)
\leq \frac{0.8 \lDK}{\sqrt{N}\ln N}
    + \Big(\frac{-1}{\sqrt{N}} + \frac{2}{N}\Big)\frac{\nK}{\ln N}
    + \frac{1.4}{\ln N}
\qquad\forall N\geq 121.
\end{equation}
The term $\psi_\K(x)$ is more delicate to deal with. In fact, we know that
\begin{equation*}
|\psi_\K(x)-x| \leq \sqrt{x}\Big[\Big(\frac{\ln   x}{2\pi} + 2\Big)\lDK
                               + \Big(\frac{\ln^2 x}{8\pi} + 2\Big)\nK
                            \Big]
\qquad\forall x\geq 100
\end{equation*}
from~\cite[Corollary~1.2]{GrenieMolteni3}, but the $\ln x$ that multiplies $\lDK$ wipes out the $\ln N$
term in the denominator of \eqref{eq:G11}, leaving only $(\lDK)/\sqrt{N}$, which does not go to zero when
we set $N = (\lDK)^2$. To overcome this difficulty we use another upper bound for $\psi_\K$ from that
paper but with a different value of a parameter controlling the trade-off between the accuracy in the $x$
and in the discriminant terms. In this way we are able to remove that factor $\ln x$ at the cost of
having only a slightly larger coefficient in the term $x$. In fact, setting $T=11$
in~\cite[Theorem~1.1]{GrenieMolteni3} we get
\begin{equation*}
\psi_\K(x)
\leq 1.1 x
     + (2.3\sqrt{x} + 1.1)\lDK
     + (\sqrt{x} + 2)\nK
     + 9\sqrt{x} + 9
\qquad\forall x\geq 2.
\end{equation*}
%
%
%
To recover this computation, we estimated the term
\begin{equation*}
\epsilon_\K(x,T) := \max\Big(0, d_\K\log x - 1.44\frac{\sqrt{x}}{T}\nK\Big)
\end{equation*}
in~\cite[Theorem~1.1]{GrenieMolteni3}, %
in which $d_{\K} = r_1+r_2-1 \leq \nK$ and $\ln x - 1.44\frac{\sqrt{x}}{T}$ has been estimated with its
maximum in $x$, i.e. with $2\ln\big(\frac{2T}{1.44}\big) - 2$, which is $< 4$ since $T = 11$. %
In this way for $N\geq 121$ we have
\begin{equation}\label{eq:G13}
\frac{\psi_\K(N)}{N\ln N}
\leq \frac{2.4 \lDK}{\sqrt{N}\ln N}
     + \Big(\frac{1}{\sqrt{N}} + \frac{2}{N}\Big)\frac{\nK}{\ln N}
     + \frac{2}{\ln N}.
\end{equation}
%
The result follows adding~\eqref{eq:G12} and~\eqref{eq:G13} in the original sum~\eqref{eq:G11}.
\end{proof}
\begin{lemma}\label{lem:G5}
Assume that $\zeta$ satisfies RH and let $N\geq 121$. Then
\begin{equation*}
\sum_{n\leq N} \Lambda(n)\big(f(n) - f(N) - (n-N)f'(N)\big)
= \ln\ln N + \gamma - \frac{3/2}{\ln N} - \frac{1/2}{\ln^2 N} + \frac{c(N)}{\sqrt{N}\ln N}
\end{equation*}
with $|c(N)|\leq 0.6$. In particular, in this range we have
\begin{equation*}
\sum_{n\leq N} \Lambda(n)\big(f(n) - f(N) - (n-N)f'(N)\big)
\leq \ln\ln N + \gamma - \frac{3/2}{\ln N}.
\end{equation*}
\end{lemma}
%
\begin{proof}
To ease the notation, we write
\begin{equation*}
 \Sigma(N) = \sum_{n\leq N} \Lambda(n)\big(f(n) - f(N) - (n{-}N)f'(N)\big)
\end{equation*}
in this proof. Using the Riemann--Stieltjes representation of the sum and integration by parts twice,
we have
\begin{align*}
\Sigma(N)
&= \int_{2^-}^N \!\! \big(f(x) - f(N) - (x{-}N)f'(N)\big)\dd\psi(x)\\
&= -\int_2^N \big(f'(x) - f'(N)\big)\psi(x) \dd x
=  \int_2^N f''(x)\psi^{(1)}(x) \dd x.
\end{align*}
Under RH, \cite[Theorem~1.1]{GrenieMolteni2} ensures that $\psi^{(1)}(x) = x^2/2 + R(x)$ with
\begin{equation}\label{eq:G14}
|R(x)| \leq 0.0462x^{3/2} + 1.838 x \qquad \forall x\geq 3,
\end{equation}
so that
\begin{align}
\Sigma(N)
&= \int_2^N f''(x)\frac{x^2}{2} \dd x + \int_2^N f''(x)R(x)\dd x.                              \notag\\
\intertext{The first integral can be evaluated explicitly, and the computation gives}
\Sigma(N)
&= \Big[f'(x)\frac{x^2}{2}\Big]_2^N - \int_2^N f'(x)x\dd x + \int_2^N f''(x)R(x)\dd x          \notag\\
&= \int_2^N f(x)\dd x + \Big[f'(x)\frac{x^2}{2}- f(x)x\Big]_2^N + \int_2^N f''(x)R(x)\dd x     \notag\\
&= \ln\ln N + c - \frac{3/2}{\ln N} - \frac{1/2}{\ln^2 N} + \int_2^N f''(x)R(x)\dd x,          \notag
\intertext{for a suitable constant $c$. The bound~\eqref{eq:G14} shows that the remaining integral converges,
hence there exists a new constant $c'$ such that}
\Sigma(N)
&= \ln\ln N + c' - \frac{3/2}{\ln N} - \frac{1/2}{\ln^2 N} + \int_N^{+\infty} f''(x)R(x)\dd x. \label{eq:G15}
\end{align}
Moreover,
\begin{align*}
\int_N^{+\infty}f''(x)x^{3/2}\dd x
 &=  \frac{3}{2}f(N)\sqrt{N} - f'(N)N^{3/2} + \frac{3}{4}\int_N^{+\infty}\frac{\dd x}{x^{3/2}\ln x}  \\
 &\leq \frac{3}{2}f(N)\sqrt{N} - f'(N)N^{3/2} + \frac{3/2}{\sqrt{N}\ln N} \\
 &= \frac{1}{\sqrt{N}\ln N}\Big(4 + \frac{1}{\ln N}\Big)
\intertext{and}
\int_N^{+\infty}f''(x)x\dd x
 &= f(N) - f'(N)N
  = \frac{1}{N\ln N}\Big(2 + \frac{1}{\ln N}\Big).
\end{align*}
Since $f''(x) \geq 0$, the bound for $R(x)$ given in \eqref{eq:G14} gives
\begin{align}
\sqrt{N}\ln N
\Big|\int_N^{+\infty} f''(x)R(x)\dd x\Big|
&\leq \sqrt{N}\ln N\int_N^{+\infty} f''(x)|R(x)| \dd x                                            \notag\\
&\leq 0.0462 \Big(4 + \frac{1}{\ln N}\Big) + \frac{1.838}{\sqrt{N}}\Big(2 + \frac{1}{\ln N}\Big)
 \leq 0.6                                                                                         \label{eq:G16}
\end{align}
for $N\geq 121$.
%
Plugging~\eqref{eq:G16} into~\eqref{eq:G15} we get the claim, but to complete the proof we have to identify
$c'$ with $\gamma$. For this purpose we first notice that
\begin{equation*}
|f(N) + (n-N)f'(N)| \ll \frac{1}{N\ln N},
\end{equation*}
so that the crude bound for $\psi(x)\ll x$ already gives
\begin{align*}
\sum_{n\leq N} \Lambda(n)&\big(f(n) {-} f(N) {-} (n{-}N)f'(N)\big)
 = \sum_{n\leq N} \Lambda(n)f(n) {+} o(1)
 = \sum_{n\leq N} \frac{\Lambda(n)}{n\ln n} {+} o(1)
\intertext{and Mertens' result gives (see proof of Theorem 2.7 in \cite{MontgomeryVaughan})}
&= \ln\ln N + \gamma + o(1).
\qedhere
\end{align*}
%
\end{proof}
\noindent
We also need the following extension to a generic field $\L$ of the result in Lemma~\ref{lem:G5}.
\begin{lemma}\label{lem:G6}
Assume that $\L\neq \Q$ satisfies GRH and let $N\geq 121$. Then
\begin{multline*}
\sum_{n\leq N} \tilde{\Lambda}_\L(n)\big(f(n) - f(N) - (n-N)f'(N)\big) \\
\leq \ln\ln N + \gamma + \ln\kL
 - \frac{3/2}{\ln N} - \frac{1/2}{\ln^2 N}
 + 2.5\frac{\lDL - \nL + 10}{\sqrt{N}\ln N}.
\end{multline*}
\end{lemma}
%
\begin{proof}
To ease the notation, we write
\begin{equation*}
\Sigma_\L(N) = \sum_{n\leq N} \tilde{\Lambda}_\L(n)\big(f(n) - f(N) - (n{-}N)f'(N)\big)
\end{equation*}
in this proof. Using the Riemann--Stieltjes representation of the sum and integration by parts twice,
we have
\begin{align*}
\Sigma_\L(N)
=  \int_2^N f''(x)\psi^{(1)}_\L(x) \dd x.
\end{align*}
As in Lemma~\ref{lem:G4}, from~\cite[Theorem~1.1]{GrenieMolteni2} we deduce that
$\psi_\L^{(1)}(x)= \frac{x^2}{2} + R_\L(x)$ with
\begin{align*}
|R_\L(x)|
\leq& (x^{3/2} + 2x + 2)\big(0.5375 \lDL - 1.0355\nL + 5.4341\big)
    + (x\ln x + 1)\nL.
\end{align*}
To simplify the computations, we use the inequality
\begin{equation*}
x\ln x + 1 \leq 0.3691(x^{3/2} + 2x + 2)
\qquad 
\forall x\geq 121,
\end{equation*}
%
producing the weaker (but nicer) bound
\begin{align}
|R_\L(x)|
\leq& (x^{3/2} + 2x + 2)\big(0.5375 \lDL - 0.6664\,\nL + 5.4341\big)  \notag\\
\leq& 0.5375 (x^{3/2} + 2x + 2)\big(\lDL - \nL + 10\big),             \label{eq:G17}
\end{align}
%
where in the second inequality we have used that $\nL\geq 2$.
%
Thus,
\begin{align}
\Sigma_\L(N)
&= \int_2^N f''(x)\frac{x^2}{2} \dd x + \int_2^N f''(x)R_\L(x)\dd x.                             \notag\\
\intertext{As we did for Lemma~\ref{lem:G5}, the first integral can be evaluated explicitly and
the second integral converges to a constant by~\eqref{eq:G17}, so that}
\Sigma_\L(N)
&= \ln\ln N + c' - \frac{3/2}{\ln N} - \frac{1/2}{\ln^2 N} - \int_N^{\infty} f''(x)R_\L(x)\dd x  \label{eq:G18}
\end{align}
for a suitable constant $c'$. In particular,
$
\int_2^N \psi_\L^{(1)}(x) f''(x) \dd x
= \ln\ln N + c' + o(1)
$.
Mertens' Theorem for the field $\L$ says that
\begin{equation*}
\prod_{\substack{\p\subseteq \OL \\ \Norm\p \leq x}} \Big(1 - \frac{1}{\Norm\p}\Big)^{-1}
\sim e^\gamma \kL \ln x ,
\end{equation*}
see \cite{GarciaLee2} or \cite{Rosen}. %
Since $\Sigma_\L(N) = - \sum_{\Norm\p\leq N} \ln\big(1 - \frac{1}{\Norm\p}\big) + o_\L(1)$, we discover
that $c' = \gamma + \ln\kL$.
%

Moreover, with two integrations by parts, for all $N\geq 121$ we have
\begin{align*}
\int_N^{\infty} &(x^{3/2}+2x+2) f''(x)\dd x \\
&=    -f'(N)(N^{3/2}+2N+2) + f(N)\Big(\frac{3}{2}\sqrt{N}+2\Big) + \frac{3}{4}\int_N^{\infty}\frac{\dd x}{x^{3/2}\ln x}\\
&\leq -f'(N)(N^{3/2}+2N+2) + f(N)\Big(\frac{3}{2}\sqrt{N}+2\Big) + \frac{3}{2\sqrt{N}\ln N}\\
&=    \frac{1+1/\ln N}{\sqrt{N}\ln N}\Big(1+\frac{2}{\sqrt{N}}+\frac{2}{N^{3/2}}\Big) + \frac{1}{\sqrt{N}\ln N}\Big(3+\frac{2}{\sqrt{N}}\Big)\\
&\leq \frac{4.62}{\sqrt{N}\ln N} .
\end{align*}
%
This evaluation and the bound for $R_\L(x)$ given in \eqref{eq:G17} yield
\begin{equation}\label{eq:G19}
\sqrt{N}\ln N
\Big|\int_N^{+\infty} f''(x)R_\L(x)\dd x\Big|
\leq 4.62 \cdot 0.5375\big(\lDL - \nL + 10\big).
\end{equation}
%
Plugging~\eqref{eq:G19} into~\eqref{eq:G18} we get the claim.
\end{proof}

\section{Proof of Theorem~\ref{th:G1}}\label{sec:G14}
To prove both theorems, we use a similar strategy. For discriminant above a given threshold, the previous
lemmas provide a proof of the statements. Below that threshold, we check directly each statement using
PARI/GP~\cite{PARI2} to compute the residue and compare it to the corresponding bound. We have retrieved
the necessary lists of fields from \cite{MegrezTables}; \cite{LMFDB} is an alternative possible source.

\subsection*{Proof of~(\ref{eq:G8})}
Let $N\geq 121$. Plugging the bounds in Lemmas~\ref{lem:G2}, \ref{lem:G3}, \ref{lem:G4}, and~\ref{lem:G5}
into the formula in Lemma~\ref{lem:G1} we get that
\begin{align*}
\ln\kK
\geq& \ln\zeta(\nK)
     - \frac{4}{\sqrt{N}\ln N}
     - \Big(\frac{3.2 \lDK}{\sqrt{N}\ln N}
            + \frac{4 \nK}{N\ln N}
            + \frac{3.4}{\ln N}\Big)                                                                     \\
    &- \Big(\ln\ln N + \gamma - \frac{3/2}{\ln N}\Big)
     - \Big(\frac{2.5\lDK}{\sqrt{N}\ln N} - \frac{3.5\nK}{\sqrt{N}\ln N} + \frac{23.5}{\sqrt{N}\ln N}\Big).
\end{align*}
The assumptions $N\geq 121$ and $\nK\geq 2$ assure that $\frac{3.5\nK}{\sqrt{N}\ln N} - \frac{4 \nK}{N\ln N}
\geq \frac{6}{\sqrt{N}\ln N}$. With this simplification the claim becomes
\begin{align}\label{eq:G20}
\ln\kK
\geq& \ln\zeta(\nK)
     - \gamma
     - \ln\ln N
     - \frac{5.7 \lDK}{\sqrt{N}\ln N}
     - \frac{1.9}{\ln N}
     - \frac{21.5}{\sqrt{N}\ln N}.
\end{align}
In this formula we chose $N=(e\lDK)^2$ and we estimate the resulting $\ln\ln N$ term with $\ln\llDK + \ln
2 + 1/\llDK$, producing the bound
\begin{align*}
\ln\kK
\geq& \ln\zeta(\nK)
     {-} \gamma
     {-} \ln 2
     {-} \ln\llDK
     {-} \frac{1}{\llDK}
     {-} \frac{1.9 {+} 5.7/e {+} 21.5/(e\lDK)}{2\ln(e\lDK)}.
\end{align*}
The last term is bounded below by $\frac{-2}{\llDK}$.\footnote{To verify this claim, put $a:= (1.9 + 5.7/e)/2$, $b :=
21.5/(2e)$ and $L := \lDK$. Then the last term is $-\frac{a + b/L}{1+\ln L}$ which is $\geq \frac{-2}{\ln
L}$ if and only if $b\frac{\ln L}{L} \leq 2 + (2-a)\ln L$. This is true, since $b\frac{\ln L}{L}\leq b/e
<2$ and $a<2$.} Therefore, 
%
%
\begin{align*}
\ln\kK
\geq& \ln\zeta(\nK)
     - \gamma
     - \ln 2
     - \ln\llDK
     - \frac{3}{\llDK} ,
\end{align*}
which is the claim in~\eqref{eq:G8}. The assumption $N\geq 121$ with $N=(e\lDK)^2$ is satisfied as soon
as $|\DK|\geq 58$,
the remaining cases are proved directly using PARI/GP~\cite{PARI2}.

For the improved bounds~\eqref{eq:G10} we use directly~\eqref{eq:G20} with respectively $N=(7\lDK)^2$
(for $h_0 = 1$), $N=(6\lDK)^2$ (for $h_0 = 100$), and $N=(5\lDK)^2$ (for $h_0 = 2\,000$).
%
%
%

\subsection*{Proof of~(\ref{eq:G9})}
Since $\tilde{\Lambda}_\K(n)\leq \nK\Lambda(n)$ and $f(n) - f(N) - (n-N)f'(N) \geq 0$, from
Lemmas~\ref{lem:G1}, \ref{lem:G2}, and~\ref{lem:G5} we get
\begin{equation*}
\ln\kK
\leq (\nK-1)\Big(\ln\ln N + \gamma - \frac{3/2}{\ln N}\Big)
     + \frac{2.5 \lDK}{\sqrt{N}\ln N} - \frac{3.5 \nK}{\sqrt{N}\ln N} + \frac{23.5}{\sqrt{N}\ln N}.
\end{equation*}
Assuming $|\DK|\geq 60\,000$, setting $N=(\lDK)^2\geq 121$ one gets
%
\begin{align*}
\ln\kK
\leq& (\nK{-}1)\big(\ln\llDK + \ln 2 + \gamma\big)
     + \frac{4 - 3 \nK/2 + \frac{23.5 - 3.5 \nK}{\lDK}}{2\llDK} .
\end{align*}
The second term is always negative when $\nK\geq 4$, but for $\nK=3$ this only occurs when $|\DK|\geq
1.9\cdot 10^{11}$, and never occurs when $\nK=2$.
%
%
Therefore, we estimate that term with something which is not too big when $\nK=2$ or $3$, but that remains
negative for $\nK\geq 4$. Our choices for the parameters give $4 + \frac{23.5 - 3.5 \nK}{\lDK}\leq 5.5$.
%
This proves that
\begin{align*}
\ln\kK
\leq& (\nK-1)\big(\ln\llDK + \ln 2 + \gamma\big) + \frac{11 - 3 \nK}{4\llDK},
\end{align*}
which is the claim in~\eqref{eq:G9} in case $|\DK|\geq 60\,000$. The validity of the claim for
$|\DK|< 60\,000$ is verified directly, using PARI/GP~\cite{PARI2}.

\section{Proof of Theorem~\ref{th:G2}}\label{sec:G15}
We have already noticed that the claim holds true in case $\L=\Q$ or $\L=\K$, hence we can assume that
$\Q \subsetneq \L \subsetneq \K$. Subtracting copies of the formula in Lemma~\ref{lem:G1} for $\K$
and $\L$, we get
\begin{equation*}
\ln(\kK/\kL)
= \sum_{n\leq N} (\tilde{\Lambda}_\K(n) - \tilde{\Lambda}_\L(n))(f(n) - f(N) - (n - N)f'(N)) + \RR_\K(N) - \RR_\L(N) ,
\end{equation*}
where the subscripts $\K$ and $\L$ indicate the field involved.

To estimate the sum we take advantage of the bound $\tilde{\Lambda}_\K \leq \nKL \tilde{\Lambda}_\L$.
%
Hence
\begin{equation*}
\ln(\kK/\kL)
\leq (\nKL-1)\sum_{n\leq N} \tilde{\Lambda}_\L(n)(f(n) - f(N) - (n - N)f'(N)) + \RR_\K(N) - \RR_\L(N).
\end{equation*}
By Lemma~\ref{lem:G6} for the sum and Lemma~\ref{lem:G2} for $\RR_\L$ and $\RR_\K$, we get
\begin{align*}
\ln(\kK/\kL)
\leq& (\nKL - 1)\Big(\ln\ln N {+} \gamma {+} \ln\kL  - \frac{3/2}{\ln N} - \frac{1/2}{\ln^2 N} {+} 2.5\frac{\lDL - \nL {+} 10}{\sqrt{N}\ln N}\Big) \\
    &+ \frac{2.5 \lDL - 3.5 \nL + 23.5}{\sqrt{N}\ln N} + \frac{2.5 \lDK - 3.5 \nK + 23.5}{\sqrt{N}\ln N}.
\end{align*}
Therefore,
\begin{align*}
\ln\kK
\leq&  (\nKL - 1)(\ln\ln N + \gamma) + \nKL\ln\kL
      + \frac{2.5 \lDK - 3.5 \nK + 20}{\sqrt{N}\ln N}      \\
    & - (\nKL - 1)\Big(\frac{3/2}{\ln N} + \frac{1/2}{\ln^2 N}\Big)
      + 2.5\nKL\frac{\lDL - \nL + 10}{\sqrt{N}\ln N}.
\end{align*}
(The constant $20$ comes from noticing that $-2.5(-\nL+10)-3.5\nL+23.5+23.5=-\nL+22\leq 20$ because
$\nL\geq 2$).
Moreover, from the formula of the discriminant in extensions we know that $|\Delta_\L|^{\nKL} \leq
|\Delta_\K|$, so this is
\begin{align}
\ln\kK
\leq&  (\nKL - 1)(\ln\ln N + \gamma) + \nKL\ln\kL
      + \frac{5 \lDK}{\sqrt{N}\ln N}                                    \notag\\
    &\qquad - \frac{\nKL - 1}{\ln N}\Big(\frac{3}{2} + \frac{1/2}{\ln N}\Big)
      - \frac{6 \nK - 25\nKL - 20}{\sqrt{N}\ln N}.                      \label{eq:G21}
\end{align}
The second line is negative for every choice of degrees $\nKL$ and $\nK$ as soon as $\L \neq \K$ and $N$
(and hence $|\DK|$) is large enough, but this threshold varies significantly as a function of $\nKL$ and
$\nK$. For this reason, we estimate it in a weaker but more convenient way.
Since $\nK = \nL \nKL \geq 2\nKL$, we have $6 \nK - 25\nKL - 20 \geq -13(\nKL-1)-33$ and the last line
in~\eqref{eq:G21} is estimated by
\begin{equation*}
 - \frac{\nKL - 1}{\ln N}\Big(\frac{3}{2} + \frac{1/2}{\ln N} - \frac{13}{\sqrt{N}}\Big)
   + \frac{33}{\sqrt{N}\ln N}.
\end{equation*}
Assume that $N\geq 190$. Then the coefficient $\frac{3}{2} + \frac{1/2}{\ln N} - \frac{13}{\sqrt{N}}$ is
$\geq 0.65$
%
and $\frac{33}{\sqrt{N}}\leq 2.4$.
%
With this bound, from~\eqref{eq:G21} we deduce that
\begin{align*}
\ln\kK
\leq&  (\nKL - 1)(\ln\ln N + \gamma) + \nKL\ln\kL
      - 0.65\frac{\nKL-1}{\ln N} + \frac{2.4}{\ln N} + \frac{5\lDK}{\sqrt{N}\ln N}.
\end{align*}
We set $N= (c\lDK)^2$ for a suitable constant $c\geq 1$ that we will choose later. Then, from the bound
$\ln\ln N \leq \ln\llDK + \ln 2 + \frac{\ln c}{\llDK}$ and the lower bound $\ln N \geq 2\llDK$ (for the
denominators), we deduce that
\begin{align*}
\ln\kK
\leq&  (\nKL - 1)(\ln\llDK + \ln 2 + \gamma) + \nKL\ln\kL
      + \frac{W(\nKL-1,c)}{\llDK}
\end{align*}
with $W(r,c)= r\ln c + 5/(2c) + (2.4-0.65\,r)/2$.
The minimum for $W(r,c)$ is attained when $c=5/(2r)$, but this choice is compatible with the assumption
$c\geq 1$ only in the first two cases $\nKL=2$, $3$, and also in those cases the condition
$N=(c\lDK)^2\geq 190$ is difficult to meet with those values, so that our choices are: $c=5$ for $\nKL=2$
(i.e., $r=1$), $c=1.3$ for $\nKL=3$ (i.e., $r=2$) and $c=1$ for $\nKL\geq 4$ (i.e., $r\geq 3$). These
choices are suboptimal but still able to produce the upper bound $W\leq 3$.
%
As a consequence, we have proved that
\begin{equation*}
\kK
\leq \kL^{\nKL}(2 e^{\gamma}\llDK)^{\nKL - 1} e^{\frac{3}{\llDK}}
\end{equation*}
when $\Delta_\K$ is such that $(5\lDK)^2\geq 190$ (if $\nKL=2$), $(1.3\lDK)^2\geq 190$ (if $\nKL=3$)
and $(\lDK)^2\geq 190$ (if $\nKL\geq 4$).
\par\smallskip \noindent %
We have $(5\lDK)^2\geq 190$ for $|\DK|\geq 16$;
%
since all fields of degree at least $4$ have $|\DK|\geq 117$, the case for $\nKL=2$ is settled.
\par\smallskip \noindent %
We have $(1.3\lDK)^2\geq 190$ for $|\DK|\geq 40261$;
%
since all fields of degree at least $9$ have $|\DK|\geq 10^8$ (by Odlyzko's bounds \cite{Odlyzko1,
OdlyzkoTables}, Table 4 entry $b=2.3$ for the discriminant),
%
the case for $\nKL=3$ is settled, except for the sextic fields with $|\DK|\leq 40\,260$ and a (necessarily
unique) quadratic subfield. We perform the check on those fields directly, using \cite{MegrezTables} to
generate the necessary list of sextic fields (64 of them with $\DK\leq 40\,260$) and PARI/GP instruction
{\tt nfsubfields} to generate the possible quadratic subfields (only 25 of them have a quadratic subfield).
\par\smallskip \noindent  %
We have $(\lDK)^2\geq 190$ for $|\DK|\geq 10^6$, and
since all fields of degree at least $8$ have $|\DK|\geq 10^6$ (again by Odlyzko's Table 4, entry
$b=2.1$), also the last case $\nKL\geq 4$ is settled.
%


\end{document}